\documentclass[11pt]{amsart}
\usepackage{amssymb}
\usepackage{amsmath}
\newcommand{\what}{\widehat}%
\newcommand{\R}{\mathbb R}%
\newcommand{\C}{\mathbb C}%
\usepackage{texdraw}
\newcommand{\N}{\mathbb N}%
\newcommand{\hc}{\mathrm c}
\newcommand{\Ad}{\mathrm {Ad}}%
\newcommand{\g}{\mathfrak g}
\newcommand{\kk}{\mathfrak k}
\newcommand{\ag}{\mathfrak a}
\newtheorem{theorem}{Theorem}[section]

\newtheorem{lemma}[theorem]{Lemma}

\theoremstyle{definition}

\theoremstyle{definition}
\newtheorem{remark}[theorem]{Remark}

\numberwithin{equation}{subsection}
\numberwithin{theorem}{subsection}

\begin{document}
\baselineskip15pt

\author[R. P. Sarkar]{Rudra P. Sarkar}
\address[R. P. Sarkar]{Stat-Math Unit, Indian Statistical
Institute, 203 B. T. Rd., Calcutta 700108, India, email:
rudra@isical.ac.in}

\title[quasi-analytic functions]{Quasi-analytic $L^p$-functions on Riemannian symmetric spaces of noncompact type, a theorem of Chernoff}
\begin{abstract} A result of Chernoff  gives sufficient condition for an $L^2$-function on $\R^n$ to be quasi-analytic. This is a generalization of the classical Denjoy-Carleman theorem on $\R$ and of the subsequent work on $\R^n$ by Bochner and Taylor. In this note we endeavour to obtain an exact analogue of the result of Chernoff for $L^p, p\in [1,2]$ functions on the Riemannian symmetric spaces of noncompact type. No restriction on the rank of the symmetric spaces and no condition on the symmetry of the functions is assumed.\end{abstract}
\maketitle

\setcounter{tocdepth}{1}

\section{Introduction}
The following result, of the genre of Denjoy-Carleman theorem, is proved in \cite{Cher}.
\begin{theorem}\label{chernoff}
Let $f$ be a $C^\infty$ function on $\R^n$. Assume that, for all integers  $m\ge 0$, $\Delta_{\R^n}^m f$ is in $L^2(\R^n)$, and that $\sum_{m=1}^\infty \|\Delta_{\R^n}^m f\|_2^{-1/2m}=\infty$, where $\Delta_{\R^n}$ is the Laplacian on $\R^n$. Suppose that all partial derivatives of $f$ vanish at $0$. Then $f$ is identically zero.
\end{theorem}
We call a set $\mathcal S$ of $C^\infty$-functions on $\R^n$ a {\em quasi-analytic} class if whenever a function $f\in \mathcal S$ and all its derivatives vanish at a point $p\in \R^n$ then $f\equiv 0$ (\cite{KP}). In this terminology  the result above gives a sufficient condition for an  $L^2$-function on $\R^n$ to be quasi-analytic. A brief description of the background of this result and an exposition on quasi-analyticity  is given at the end of this section.
Very recently following two theorems (\cite{BPR1}, \cite{BPR2}) endeavored to generalize Theorem \ref{chernoff} for Riemannian symmetric spaces of noncompact type. A prototypical  example of such spaces is the real hyperbolic space $\mathbb H^n$. We recall that a  Riemannian symmetric space of noncompact type can be realized as a quotient space $G/K$, where $G$ is a connected noncompact semisimple Lie group with finite centre and $K$ is a maximal compact subgroup of $G$.  The group $G$ and its subgroup $K$ acts naturally (from left) on $G/K$.  Let  $\Delta$ be the Laplace-Beltrami operator on $G/K$ and $\mathbf{D}(G/K)$ be the set of left $G$-invariant differential operators on $G/K$.
\begin{theorem}[\cite{BPR1}]\label{BPR1}
Let $f\in C^\infty(G/K)$ be such that $\Delta^mf\in L^2(G/K)$, for all $m\in \N\cup \{0\}$ and
\[\sum_{m=1}^\infty \|\Delta^m f\|_2^{-\frac 1{2m}}=\infty.\]
If $f$ vanishes on any nonempty open set in $G/K$ then $f$ is identically zero.
\end{theorem}
The hypothesis of this theorem differs from that of Theorem \ref{chernoff} in the vanishing condition. Indeed, the condition implies in particular that {\em all derivatives} of $f$ are zero at every point of an open set. An important subclass of functions on $G/K$ are those which are invariant under left $K$-action. When $G/K=\mathbb H^n$, then these $K$-invariant functions are {\em radial}, i.e. the value of the function at a point depends only on the distance of the point from the origin of $G/K$.
In the next theorem the  vanishing condition is only at a point like Theorem \ref{chernoff}, but this is achieved at the cost of restricting to the class of $K$-invariant functions on $G/K$, denoted by $C^\infty(G//K)$.
\begin{theorem}[\cite{BPR2}]\label{BPR2}
Let $f\in C^\infty(G//K)$ and $p\in [1,2]$. Suppose  $\Delta^mf\in L^p(G/K)$, for all $m\in \N\cup \{0\}$ and
\[\sum_{m=1}^\infty \|\Delta^m f\|_p^{-\frac 1{2m}}=\infty.\]
If there exists $x_0\in G/K$, such that $Df(x_0)=0$ for all $D\in \mathbf{D}(G/K)$  then $f$ is identically zero.
\end{theorem}
Purpose of this note to prove the following analogue of Theorem  \ref{chernoff} for $G/K$, where $\Delta$ and $\mathbf{D}(G/K)$ are as defined above and $d_r(u)$ is the right $G$-invariant differential operator defined by elements $u$ of the universal enveloping algebra $U(\g)$ of $G$. (See Section 2 for details.)

\begin{theorem}
\label{main-result} Let $p\in [1, 2]$ be fixed. Suppose that a function $f\in C^\infty(G/K)$ satisfies the following conditions:

(i) $d_r(u)f\in L^p(G/K)$ for all  $u\in U(\mathfrak g)$,

(ii) $\sum_{n\in \N} \|\Delta^n f\|_p^{-1/2n}=\infty$ and

(iii) for a fixed point $x_0\in G/K$,  $d_r(u)f(x_0)=0$ for all $u\in U(\mathfrak g)$.

Then $f=0$.
\end{theorem}
It follows from the Taylor's theorem that a nonzero  real analytic function on $G/K$  cannot satisfy condition (iii) above. Thus the Theorem \ref{main-result} asserts that a function $f$ on $G/K$ which satisfies condition (i) and (ii) of the hypothesis is {\em quasi-analytic} in the sense that like a nonzero real analytic function, this function $f$ also cannot  satisfy (iii).

In the hypothesis of the theorem above,  precisely those differential operators of $G$ are used, which in particular preserves the right $K$-invariance of a function $f$ on $G$, and hence are relevant for functions on $G/K$. It is worth pointing out that $\{d_r(u) \mid u\in U(\g)\}$ includes $\mathbf{D}(G/K)$. On the other hand the theorem is not true if we restrict only to the differential operators in $\mathbf{D}(G/K)$. A counter example to establish this is  given in \cite{BPR1}.
Thus we may consider Theorem \ref{main-result} an exact analogue of Theorem \ref{chernoff} for $G/K$. (See Section 4 for a discussion on this.) If we assume that $f$ is $K$-biinvariant then this theorem reduces to Theorem \ref{BPR2}, as expected (see Remark \ref{reduce-to-biinvariant}).
For convenience we shall call conditions (ii) and (iii) in the hypothesis respectively as {\em Carleman-type condition} and the {\em vanishing condition}. Section 2 contains all preliminaries required for this paper. Theorem \ref{main-result} is proved in Section 3.
\subsection{Background} If $f$ is a real analytic function on an interval $(a, b)\subseteq \R$, then $f$ and all of its derivatives cannot vanish at a point $x_0\in (a, b)$. On the other hand, by definition, as the Taylor series of an analytic function has to converges to the function, the formula for the remainder term shows that the derivatives of $f$ cannot grow too fast. These observations intrigue one to find an appropriate growth condition on $C^\infty$-functions to define a space which accommodates functions which are not necessarily analytic, yet their nonzero members cannot vanish at a point along with all of its derivatives. For this last property which they share with real analytic functions, members of this space are called the {\em quasi-analytic} functions. The seminal result of Denjoy and Carleman (\cite{Den-21, Carleman}) provides the precise formulation of this on $\R$. See \cite{Rudin, Cohen} for  a simple proof and exposition. Subsequently Bochner and Taylor \cite{B-T-1939} extended this result to $\R^n$ and other spaces, in which various special differential operators were constructed to replace $d/dx$ used in the Denjoy-Carleman's theorem. In all these results $L^\infty$-norms of the derivatives were considered to restrict the growth of the sequence of derivatives.
Some important developments in the intervening period are the study of analytic vectors of elliptic operators on Lie groups by  Nelson \cite{Nelson} and subsequent work on quasi-analytic vectors by Nussbaum \cite{Nuss1965, Nuss1969} in the set up of operator theory on Hilbert spaces, which was furthered by the work of de Jeu \cite{de-Jeu}. Our final point of departure from these classical results is the result of Chernoff (Theorem \ref{chernoff}), who obtained this as a part of his study of the relation between these operator theoretic developments on quasi-analytic vectors and quasi-analytic functions. Theorem  \ref{chernoff} differs from the result of Bochner and Taylor in a number of ways. Most importantly, Carleman-type condition  involves only (powers of) Laplacian and the $L^2$-norm replaces  $L^\infty$-norm. Theorem \ref{BPR1} and \ref{BPR2} are the first two attempts to obtain versions of this result  for Riemannian symmetric spaces. A recent paper \cite{GMT} obtains a different version for $L^2$-functions on rank one symmetric spaces, where the vanishing condition is taken on every $k\in K$. Our aim  here is to offer a generalization of Theorem \ref{chernoff} to the symmetric spaces of noncompact type of any rank and without any restriction of $K$-invariance on the functions.

\section{Preliminaries}
The prerequisites of semisimple Lie groups and Riemannian symmetric spaces  are standard and widely available, e.g. in \cite{GV, Wallach, Helga-2, Helga-3}. To make this note self contained we shall gather them here without elaboration, and in the process will establish the notation, most of which are also standard.

A Riemannian symmetric space of noncompact type can be realized as a quotient space $G/K$, where  $G$ is a connected  noncompact semisimple Lie group with finite centre and $K$ is a maximal compact subgroup of $G$. The Group $G$ and hence  $K$ acts naturally on $G/K$. Let us denote the origin $\{K\}$ of $G/K$ by $o$.

A function $f$ on $G$ is called right (respectively left) $K$-invariant if $f(gk)=f(g)$ (respectively $f(kg)=f(g)$) for all $g\in G$ and $k\in K$. If $f$  is both left and right $K$-invariant then it is called a $K$-biinvariant function.   We shall frequently identify right $K$-invariant functions on $G$ with functions on $G/K$. Through this identification  a left $K$-invariant function on $G/K$ is a $K$-biinvariant function on $G$. The subset of the left $K$-invariant functions in a function space $\mathcal F(G/K)$ will be denoted by $\mathcal F(G//K)$.

Let $\mathfrak g_0$ and $\kk_0$ be the Lie algebras of $G$ and $K$ respectively and $\mathfrak p$ be the orthogonal complement of $\kk_0$ in $\g_0$ with respect to the Cartan-Killing form of $\g_0$. Let $\ag_0$ be a maximal abelian subspace of $\mathfrak p$, whose dimension (which is independent of its choice) is the real rank of $G$ and the rank of the symmetric space $G/K$.
We shall assume that the rank of $G/K$ is $l$ and  $\dim G/K$ is $n$. Let $\Sigma$ be the restricted root system of the pair $(\g_0, \ag_0)$, and $W$ the associated Weyl group. For a root $\alpha\in \Sigma$, let $\g_\alpha$ be the associated root space and let $\dim \g_\alpha$ be $m_\alpha$, which is called the multiplicity of the root $\alpha$.
Let $(\ag_0)_+\subset \ag_0$ be a positive Weyl chamber  and $\Sigma^+$ the corresponding set of positive roots. Let $\mathfrak n_0=\sum_{\alpha\in \Sigma^+} \g_\alpha$. Then  $N=\exp \mathfrak n_0$ is a simply connected nilpotent Lie group and  $A=\exp \mathfrak a_0$ is an abelian group. Let $M$ be the centralizer of $A$ in $K$. We have the analytic diffeomorphism $(k, a, n)\mapsto kan$ from $K\times A\times N$ to $G$. The corresponding decomposition of $G$ is called the Iwasawa decomposition: $G=K \exp \ag_0 N$. Using this decomposition an element $g\in G$ can be uniquely written as $g=k \exp(H(g)) n$ where $k\in K, n\in N$ and $H(g)\in \ag_0$. The group $G$ also has the Cartan decomposition $G=K (\exp \overline{(\ag_0)_+}) K$.
Let $\ag_0^\ast$ be the (real) dual space of $\ag_0$ and let $\mathfrak a^\ast$ be its complexification. The half-sum of the positive roots counted with multiplicities is denoted by $\rho$ and is given by $\rho= \frac12 \sum_{\alpha\in \Sigma^+} m_\alpha \alpha \in \ag_0^\ast$.

The Killing form on $\g_0$ restricts to a positive definite form on $\ag_0$, which  induces a positive inner product on $\ag_0$ and hence on $\ag_0^\ast$. Equipped with these inner products we identify them with $\R^l$. The norm corresponding to this inner product  on $\ag_0^\ast$ is denoted by $|\cdot|$.
The Killing form endows $G/K$ with both a natural $G$-invariant Riemannian metric, hence a Laplace-Beltrami operator $\Delta$ and a corresponding $G$-invariant measure (denoted by $dx$).
In the Cartan decomposition $x=k_1 \exp Hk_2$, the Haar measure $dx$ of $G$ is given by
\begin{equation}\label{Haar} dx=  c\, \delta(H)\, dk_1 dH dk_2\end{equation} for a constant $c$ and
\[\delta(H)=\Pi_{\alpha \in \Sigma^+} (\sinh \alpha(H))^{m_\alpha}\sim O(e^{2\rho(H)}),\] where $dH$ is the Lebesgue measure on  $\overline{(\ag_0)_+}$ and $dk$ is the normalized Haar measure on $K$.

Let $\mathfrak g, \ag, \kk$ be the  complexifications of $\g_0, \ag_0, \kk_0$ respectively. Then $\mathfrak g, \ag, \kk$  are  Lie algebras over $\C$.  Let $U(\g)$ and $U(\ag)$ be the universal enveloping algebras of $G$  and $A$ respectively.
Let $\Ad$ denote the adjoint representation of $G$ on its Lie algebra $\mathfrak g_0$. Through its restriction to $K$, the  group $K$ also acts on $\g_0$ by adjoint action, hence on $\g$, and  extends to a representation of $U(\g)$. Let \[U(\g)^K=\{u \in U(\g)\mid \Ad(k) u = u \text{ for all } k\in K\}.\]
Every element of $\g_0$, hence of $U(\g)$ defines a left-invariant and a right-invariant differential operator denoted by $d_l$ and $d_r$ respectively which are defined as follows. For a function $f\in C^\infty(G)$ and $u=X_1X_2\ldots X_r, \, X_i\in \g_0$  and $x\in G$, we define (\cite[p. 84]{GV})
\[(i) \ \ \ \ d_l(u) = f(x;u)=(\partial^r/\partial t_1 \partial t_2\ldots \partial t_r)_0 f(x \exp t_1X_1\cdots \exp t_rX_r), \]
\[(ii) \ \ \ \ \ d_r(u)= f(u; x)=(\partial^r/\partial t_1 \partial t_2\ldots \partial t_r)_0 f(\exp t_1X_1\cdots \exp t_rX_r x),\]  (the suffix $0$ in the right hand sides denote that the derivatives are taken at $t_1=\cdots=t_r=0$) and then extend this definition  to all $u\in U(\g)$.
The left-invariant  differential operators $d_l(u)$ defined by the elements of $U(\g)^K$, naturally descends to left-invariant differential operators on $G/K$. Let $\kk U(\g)$ denotes the ideal generated by $\kk$ in $U(\g)$. If $u\in \kk  U(\g)\cap U(\g)^K$, then $d_l(u)$ acts trivially on $C^\infty(G/K)$, hence the set of left invariant differential operators on $G/K$, denoted by $\mathbf{D}(G/K)$ can be identified with $U(\g)^K/\kk U(\g)\cap U(\g)^K$. We recall that $\mathbf{D}(G/K)$ is a commutative algebra which contains the Laplace-Beltrami operator $\Delta$ and if the rank of $G/K$ is one, then $\mathbf{D}(G/K)$ is generated by  $\Delta$. For an element $u\in U(\g)^K$, the projection of $d_l(u)$  on $\mathbf{D}(G/K)$ will be  denoted by $D(u)$.
For $f\in C^\infty(G/K)$ and $u\in U(\g)^K$, we have $d_r(u)f=d_l(u)f=D(u)f$ (\cite[p. 52-53]{GV}). It is easy to verify that for any $f\in C^\infty(G/K)$ and $u\in U(\g)^K$,
 $(d_r(u)f)_0= d_r(u) f_0$, where $f_0$ denotes the $K$-biinvariant component of  $f$, precisely $f_0(x)=\int_K f(kx) dk$.

 Since $\exp \mathrm{Ad}(g) X = g\exp X g^{-1}$ we have $f(g\exp tX)=f(\exp t(Ad(g)X) g)$ and $f(\exp tXg)=f(g\exp t(Ad(g^{-1})X))$, for any $f\in C^\infty(G)$, $X\in \g$ and $g\in G$. This defines a transition between right and left invariant derivatives: $d_r(u)f(g)=d_l(\mathrm{Ad}(g^{-1})u f(g)$. See e.g. \cite[p. 84]{GV} for details.


Since $\ag$ is abelian, $U(\ag)$ is viewed as the set of complex polynomials on $\ag_0^\ast$.
The Weyl group $W$ acts on $\ag_0^\ast$, hence on the polynomial algebra $U(\ag)$. The set of $W$-fixed polynomials in $U(\ag)$ is denoted by $U(\ag)^W$.
There is a surjective isomorphism $\Gamma: \mathbf D(G/K) \to U(\ag)^W$. For $\lambda\in \ag^\ast, k\in K$, the function $e_{\lambda, k}(x)=e^{(i\lambda+\rho) H(x^{-1}k)}$ on $G/K$ is a joint eigenfunction of elements of $\mathbf D(G/K)$: $De_{\lambda, k}=\Gamma(D)(i\lambda) e_{\lambda, k}$ and  in particular $\Gamma(\Delta)(i\lambda)=-(|\lambda|^2+|\rho|^2)$.

Let $(\pi, H)$ be a unitary representation of $G$. A vector $v\in H$ is said to be a $C^\infty$-vector for $\pi$ if $g\mapsto \pi(g)v$ from $G$ to $H$ is a  $C^\infty$-function. The set of  $C^\infty$-vectors in $H$, denoted by $H_\infty$ is dense in $H$.
For $v\in H_\infty$ and  $X\in \g_0$ we define
\[\pi(X) v= \frac{d}{dt} \pi(\exp tX) v|_{t=0}.\]
For $X_1, X_2\in \g_0$, $X=X_1+iX_2\in \g$ and $v\in H_\infty$, we set $\pi(X) v = \pi(X_1)v+i\pi(X_2)v$. Then $\pi(X) H_\infty\subseteq H_\infty$ and $\pi([X_1, X_2])=[\pi(X_1), \pi(X_2)]$. That is $(\pi, H_\infty)$ is a representation of $\g$.
For $X_1, X_2$ and $X$ as above, we define $\overline{X}=X_1-iX_2$.
 For an element $u\in U(\g)$, its adjoint $u^\ast\in U(\g)$ is defined through the following rules:
  \[1^\ast = 1, X^\ast=-\overline{X} \text{ for }X\in \g \text{ and for }u_1, u_2\in U(\g), (u_1u_2)^\ast=u_2^\ast u_1^\ast.\] Then $\langle \pi(u) v, w\rangle=\langle  v, \pi(u^\ast) w\rangle$ for $v, w\in H_\infty$.

 A vector $v\in H_\infty$ is called $K$-finite for the representation $\pi$ if $\{\pi(k)v \mid k\in K\}$ spans a finite dimensional vector subspace of $H_\infty$. The set of all $K$-finite vectors in $H_\infty$ is denoted by $H_F$.
For $X\in \g$, $\pi(X) H_F\subset H_F$. It is known ( \cite[5.3-5.5]{Wallach}) that if an irreducible representation $(\pi, H)$ of $G$ is unitary  then the corresponding representation $(\pi, H_F)$ is also an irreducible representation of the Lie algebra $\g_0$.

Henceforth, we shall restrict our attention to the unitary class-1 principal series representations $\pi=\pi_\lambda, \lambda\in \ag_0^\ast$, realized in the compact picture, i.e.  the carrier space is $H=L^2(K/M)$ and the action of $\pi_\lambda$ on $L^2(K/M)$ is given by (\cite[p. 102]{GV}),
\[(\pi_\lambda(x)f)(k)=e^{-(i\lambda+\rho)H(x^{-1}k)} f(K(x^{-1}k),\ \ f\in L^2(K/M), x\in G, k\in K.\] For almost every $\lambda \in \ag_0^\ast$ with  respect to the Plancherel measure (described below), $\pi_\lambda$ is irreducible.

Let $\what{K}$ be the set of equivalence classes of irreducible unitary representations of $K$. For a representation $(\delta, V_\delta) \in \what{K}$, let
 \[V_\delta^M=\{v\in V_\delta \mid \delta(m)v=v \text{ for all } m\in M\}.\]
 Let $\what{K}_M$ denote the set of representations $(\delta, V_\delta) \in \what{K}$ for which the subspace $V_\delta^M\neq \{0\}$. Let $d(\delta)=\dim V_\delta, \ell(\delta)=\dim V_\delta^M$. We fix an orthonormal basis $v_1, \ldots, v_{d(\delta)}$ of $V_\delta$ such that $\{v_1, \ldots, v_{\ell(\delta)}\}$ span $V_\delta^M$ and consider the following matrix coefficients of $\delta(k), k\in K$:
 \[Y_\delta^{j, i}=\langle v_j, \delta(k)v_i\rangle, 1\le j\le d(\delta), 1\le i\le \ell(\delta).\] Then the set
 \[\{\sqrt{d_\delta}Y^{j,i}_\delta \mid \delta\in \what{K}_M, 1\le j\le d(\delta), 1\le i\le \ell(\delta)\}\] forms a countable orthonormal basis  of $L^2(K/M)$.  Any $K$-finite function in $L^2(K/M)$ is a finite linear combination of such matrix coefficients. We enumerate this orthonormal basis as $e_0, e_1, e_2, \ldots$ where $e_0$ is the only matrix coefficient of the trivial representation $\delta_0$ of $K$, hence is the constant function $1$ on $K$.
 The matrix coefficients of the principal series representations $\pi_\lambda$ with respect to this orthonormal basis, which are relevant for functions on $G/K$ are $\phi^i_\lambda(x)=\langle\pi_\lambda(x) e_0, e_i\rangle$.
In particular, $\phi^0_\lambda(x)=\langle\pi_\lambda(x) e_0, e_0\rangle =\int_K e^{-(i\lambda+\rho)H(x^{-1}k)} dk$ is the elementary spherical function denoted by $\varphi_\lambda$.
The following properties of  $\varphi_\lambda$ and  $\phi^i_\lambda$  will be used
(see \cite{HC-76, GV, Helga-2, Ank1}):

(i) $\varphi_\lambda$ is a $K$-biinvariant function  on $G$,

(ii) $\varphi_\lambda=\varphi_{w\lambda}$ for $\lambda\in \ag_0^\ast$ and $w\in W$,

(iii) $\varphi_\lambda(x^{-1})=\varphi_{-\lambda}(x)= \overline{\varphi_\lambda(x)}$ for $\lambda\in \ag_0^\ast$  and $x\in G/K$,

(iv) $\varphi_\lambda(u^\ast; x)=\overline{\varphi_\lambda(x^{-1}; u)}$ for $\lambda\in \ag_0^\ast$, $u\in U(\g)$  and $x\in G/K$,

(v)  $|\varphi_\lambda(x)|\le \varphi_0(x)\le 1$, for all $\lambda\in \ag_0^\ast$ and $x\in G/K$,

(vi) $ e^{-\rho(H)}\le \varphi_0(\exp H)\le C (1+\|H\|)^c e^{-\rho(H)}, H\in \overline{(\ag_0)_+}$ for some positive constants $C, c$,

(vii) $|\phi^i_\lambda(u_1; x; u_2)|\le C (1+|\lambda|)^{\deg u_1+\deg u_2} \varphi_0(x)$, for $u_1, u_2\in U(\g)$,
$\lambda\in \ag_0^\ast$ and $i=0,1,2,\ldots$,

(viii) $\phi_\lambda^i$ are joint eigenfunctions of $D\in \mathbf{D}(G/K)$ and $\Delta \phi_\lambda^i=-(|\lambda|^2+|\rho|^2) \phi_\lambda^i$, for all $\lambda\in \ag_0^\ast$.

It can be verified that (vi), (vii)  above and \eqref{Haar} imply that  $\phi^i_\lambda, \lambda\in \mathfrak a_0^\ast$ is in $L^{p'}(G/K)$ for any $1\le p<2$. Here and everywhere $p$ and $p'$ are related by $1/p+1/p'=1$.  Property (iv) follows from the definition of $\varphi_\lambda$ and (viii) follows from the corresponding properties of  the function $e_{\lambda, k}$ mentioned above.

For a function $f$ on $G/K$, its spherical Fourier transform $\what{f}$ is defined as
\[\what{f}(\lambda) = \int_{G/K} f(x) \varphi_{-\lambda}(x)dx,\] whenever the integral makes sense.
Thus $\what{f}(\lambda) =\langle f, \varphi_\lambda\rangle_{L^2(G/K)}$. It  is clear that $\what{f}$ is $W$-invariant and $\|\what{f}\|_{L^\infty(\ag_0^\ast, |\hc(\lambda)|^{-2}d\lambda)}\le \|f\|_1$. Here  $\hc(\lambda)$ is the Harish-Chandra $\hc$-function and $|\hc(\lambda)|^{-2} d\lambda$ is the Plancherel measure on $\ag_0^\ast$, where $d\lambda$ is the Lebesgue measure on $\mathfrak a_0^\ast$ identified with $\R^l$. Plancherel measure satisfies the estimate (\cite{Helga-2}),
\begin{equation}
\label{plancherel-estimate}
|\hc(\lambda)|^{-2} \le C (1+|\lambda|)^{\dim \mathfrak n_0}, \text { for } \lambda\in (\ag_0^\ast)_+
\end{equation}
 for some positive constant $C$.
For a function $f \in L^p(G//K), 1\le p \le 2$, satisfying $\what{f}\in L^1(\ag_0^\ast, |\hc(\lambda)|^{-2} d\lambda)$, we have the Fourier inversion formula (\cite[Theorem 3.3]{Stanton-Tomas} \cite[Theorem 5.4]{NPP}:
\begin{equation}\label{inversion}
f(g)=|W|^{-1} \int _{\ag_0^\ast} \what{f}(\lambda) \varphi_{\lambda}(g) |\hc(\lambda)|^{-2} d\lambda,
\end{equation} where $|W|$ is the cardinality of $W$.
 The map $f\mapsto \what{f}$ is an isometry from $L^2(G//K)$ to $L^2(\ag_0^\ast, |\hc(\lambda)|^{-2} d\lambda)^W$, the subspace of $W$-invariant functions in $L^2(\ag_0^\ast, |\hc(\lambda)|^{-2} d\lambda)$.

For a suitable function $f$ on $G/K$, similarly we define,
\begin{equation} \label{FT-i}
\what{f_i}(\lambda) = \int_{G/K} f(x) \phi_{-\lambda}^i(x)dx,
\end{equation} whenever the integral makes sense and we have
\begin{equation}
\label{HY-end-point-1}
\|\what{f_i}\|_{L^\infty(\ag_0^\ast, |\hc(\lambda)|^{-2}d\lambda)}\le C \|f\|_1,
 \end{equation} for some constant $C>0$.
Using the Plancherel theorem (\cite{Helga65, Helga70}) on $G/K$ we have for $i=0, 1,2, \ldots$,
\[\|\what{f_i}\|_{L^2(\ag_0^\ast, |\hc(\lambda)|^{-2} d\lambda)} \le C\|f\|_2,\]
and interpolating this with \eqref{HY-end-point-1} we have the Hausdorff-Young inequality,
\begin{equation}
\label{HY-general}
\|\what{f_i}\|_{L^{p'}(\ag_0^\ast, |\hc(\lambda)|^{-2} d\lambda)} \le C\|f\|_p, 1\le p\le 2, i=0, 1,2, \ldots,
\end{equation} for some constant $C>0$.

\section{Proof of Theorem \ref{main-result}}
We need to use the following extension of Carleman's theorem  on determinacy of measure (see  \cite[Theorem 2.3]{de-Jeu}).
\begin{lemma} \label{Carleman-lem}
Let $\mu$ be a finite Borel measure on $\R^n$ such that for all $m\in \N$ and $1\le j\le n$ the quantities $S_j(m)$ defined by
\[\int_{\R^n} |\lambda_j|^m d\mu(\lambda)\] are finite. If for $j\in \{1,2,\ldots, n\}$, the sequence $\{S_j(2m)\}_{m=1}^\infty$ satisfies the Carleman's condition
\[\sum_{m\in \N} S_j(2m)^{-1/2m}=\infty\] then the set of polynomials is a dense subspace of $L^1(\R^n, d\mu)$.
\end{lemma}
We also need the following technical result (\cite[Lemma 3.3]{BPR1}).
\begin{lemma} \label{BPR1-3.3}
Let $\{a_n\}$ be a sequence of positive numbers such that the series $\sum_{n\in \N} a_n$ diverges. Then given any $m\in \N$, $\sum_{n\in \N} a_n^{1+\frac mn}=\infty$.
\end{lemma}

A crucial part of the proof of Theorem \ref{main-result} is taken out as the following lemma.  For a locally integrable function $F$ on $G/K$, we define its $K$-invariant component $F_0(x)=\int_KF(kx) dk$.
\begin{lemma} \label{lem-1} Let $1\le p\le 2$ be fixed. Suppose that for a function  $f\in C^\infty(G/K)$, $d_r(u)f\in L^p(G/K)$ and $(d_r(u)f)_0=0$ for all $u\in U(\mathfrak g)$. Then  $f=0$.
\end{lemma}
\begin{proof}
We fix a $\lambda\in \mathfrak a_0^\ast$ such that the representation $(\pi_{-\lambda}, L^2(K/M))$ is irreducible.
We note that, due to the $K$-invariance of $\varphi_{-\lambda}$,
\begin{align} \label{decom-0}
&\int_G (d_r(u)f)_0(x)\, \varphi_{-\lambda}(x)\, dx\\ \nonumber
=&\int_G d_r(u)f(x)\, \varphi_{-\lambda}(x) \,dx \\ \nonumber
=&\int_G f(x) \, \, d_r(u^\ast)\varphi_{-\lambda}(x)\, dx
\end{align}
Since $\pi_{-\lambda}$ is unitary, we further have,
\begin{equation}\label{decom-1}
d_r(u^\ast)\varphi_{-\lambda}(x)=\langle \pi_{-\lambda}(u^\ast) \pi_{-\lambda}(x)e_0, e_0\rangle=\langle \pi_{-\lambda}(x)e_0, \pi_{-\lambda}(u) e_0\rangle. \end{equation}
Putting \eqref{decom-1}  in \eqref{decom-0}   we have
\begin{align} \label{step-1-2}
\int_G (d_r(u)f)_0(x)\, \varphi_{-\lambda}(x)\, dx = & \int_G f(x) \langle \pi_{-\lambda}(x)e_0, \pi_{-\lambda}(u) e_0\rangle dx
\end{align}
The hypothesis $(d_r(u)f)_0=0$ thus implies that
\begin{align}\label{step-1-3}\int_G f(x) \langle \pi_{-\lambda}(x)e_0, \pi_{-\lambda}(u) e_0\rangle dx = & \langle \int_G f(x) \, \pi_{-\lambda}(x) e_0\,  dx, \pi_{-\lambda}(u) e_0\rangle=0.
\end{align}
  Since $(\pi_{-\lambda}, L^2(K/M))$ is unitary and irreducible, it  induces an irreducible representation of $\mathfrak g_0$ on $L^2(K/M)_F$
    (see Section 2), hence  $e_0$ is a cyclic vector for the representation $(\pi_{-\lambda}, L^2(K/M)_F)$ of $\g_0$. Since \eqref{step-1-3} is true for all $u\in U(\mathfrak g)$, we have $\langle \int_G f(x)\pi_{-\lambda}(x)e_0\, dx, e_i\rangle=0$ for $i=0,1,2,\ldots$, hence  $\int_G  f(x)\pi_{-\lambda}(x)e_0\, dx =0$.
 As $f$ right-$K$-invariant, we also have, $\int_G f(x)\pi_{-\lambda}(x)e_i\, dx =0$  for $i=1,2,\ldots$.
  This is true for almost every $\lambda\in \mathfrak a_0^\ast$, since for almost every $\lambda\in \ag_0^\ast$, $\pi_\lambda$ is irreducible.  Hence  $f=0$.
\end{proof}

We recall that (see Section 2) for $u\in U(\g)$ and $x\in G$, $\varphi_\lambda(u^\ast; x)= (d_r(u^\ast)\varphi_\lambda)(x)=\overline{\varphi_\lambda(x^{-1};u)}$.
We have the following formula (\cite[Proposition 4.6.2.]{GV}) which by the observation above and in our  parametrization can be written as:
\begin{align} \label{derivative-phi-1}
(d_r(u^\ast)\varphi_\lambda)(x) = &\sum_{s\in \Lambda} \alpha_{s}(\lambda-\rho) \int_K  e^{-(i\lambda+\rho) H(x^{-1}k)} F_s(k) dk \\
= &\sum_{s \in \Lambda} \alpha_{s}(\lambda-\rho)\langle \pi_\lambda(x) e_0, \bar F_s\rangle,\,\, x\in G\nonumber
\end{align}
where $\Lambda$ is a finite set, $F_s\in C^\infty(K/M)$, $\alpha_{s}\in U(\mathfrak a)$, $(\lambda-\rho)\in \mathfrak a_0^\ast$ and thus $\alpha_{r}(\lambda-\rho)$ is a polynomial in $\lambda$ (see Section 2).
Since the left hand side of the equation above is $(d_r(u^\ast)\varphi_\lambda)(x)=\langle \pi_\lambda(x) e_0, \pi_\lambda(u) e_0 \rangle$, and as $\pi_\lambda(u) e_0\in H_F$,  it follows that  $F_s, \, s\in \Lambda$ are $K$-finite functions. Expanding  each $F_s$ as a finite linear combination
of the orthonormal basis elements $e_i, i=0,1,2,\ldots$ of $L^2(K/M)$, rearranging the polynomials and calling them $P_i$,  we arrive at
\begin{equation}\label{derivative-phi}
(d_r(u^\ast)\varphi_\lambda)(x)=\sum_{i\in \Lambda} P_i(\lambda)\, \varphi_\lambda^i(x).
\end{equation}

We are now ready to complete the proof of Theorem \ref{main-result}.
\begin{proof}[Proof of Theorem \ref{main-result}] We divide the proof in a few steps.
\vspace{.2in}

\noindent{\em Step 0:}
For $g\in G$, let $\ell_g:x\mapsto gx$ denote the left translation on $G/K$. Thus $\ell_gf(x)=f(gx)$.
Since $\Ad(g):\mathfrak g_0\to \mathfrak g_0$ is a linear map
and $\exp \Ad(g) X = g (\exp X) g^{-1}, g\in G, X\in \g_0$, we have,
\[f(\exp (t \Ad(g)X) x_0)=f(\exp(\Ad(g) tX) x_0)=f(g\exp(tX) g^{-1}x_0).\]
Hence,
\[(d_r(\Ad(g) X)f)(x_0)= \frac{\partial}{\partial t} f(g\exp (tX) g^{-1}x_0)|_{t=0} = (d_r(X)(\ell_gf)) (g^{-1}x_0)\] and
extending to elements  of $U(\g)$,  for $u\in U(\g)$ and $g\in G$,

\begin{equation}
\label{adjoint-rep}
(d_r(\Ad(g) u)f)(x_0)= (d_r(u)(\ell_gf)) (g^{-1}x_0). \end{equation}

The condition $d_r(u)f(x_0)=0$ for all $u\in U(\g)$ implies that for any fixed $g\in G$,  $d_r(\mathrm{Ad}(g)u)f)(x_0)=0$ for all $u\in U(\g)$. Since $G$ acts transitively on $G/K$, we take $g\in G$ such that $g^{-1}x_0=o$, (i.e. $x_0=gK$). Then for this fixed $g$,
by \eqref{adjoint-rep} $d_r(u)(\ell_gf)(o)=0$.
Thus the function $\ell_gf$ satisfies the vanishing condition (iii) at origin. It is clear that $\ell_gf$  satisfies conditions (ii) of the hypothesis, because $\Delta$ commutes with left translations and the measure $dx$ on $G/K$ is left $G$-invariant. We also have for all $u\in U(\g)$,
\[d_r(u)\ell_g f(x)=\ell_g (d_r(\mathrm{Ad}(g)) u f)(x),\]  and by the hypothesis $d_r(\mathrm{Ad}(g)) u f)\in L^p(G/K)$. Hence  $d_r(u)(\ell_g f)\in L^p(G/K)$.
That is $\ell_g f$ satisfies condition (i). Thus we can substitute $f$ by $\ell_g f$ and without loss of generality we can and shall assume that $x_0=o$. Thus the modified condition (iii) is $d_r(u)f(o)=0$ for all $u\in U(\g)$.

\vspace{.2in}

\noindent{\em Step 1:} We recall that for a function $F\in C^\infty(G/K)$,  $F_0(x)=\int_KF(kx) dk$.  It is clear that $F_0(o)=F(o)$.
If $v\in U(\g)^K$, then $(d_r(v)F)_0= d_r(v) F_0$ (see Section 2).
Therefore,  condition (iii) in the hypothesis (modified in Step 0)  implies that
\[d_r(v)(d_r(u)f)_0(o)= (d_r(v) d_r(u)f)_0(o)=0  \text{ for all } v\in U(\g)^K  \text{ and } u\in U(\mathfrak g).\]
\vspace{.2in}

\noindent{\em Step 2:} We fix a $u \in U(\mathfrak g)$. Let $h=(d_r(u)f)_0$. Since  for $v\in U(\g)^K$, $d_r(v)h=d_l(v)h=D(v)h$ (see Section 2), by Step 1, $Dh(o)=0$ for all $D\in \mathbf D(G/K)$. It is also clear that $Dh\in L^p(G//K)$ for all $D\in \mathbf D(G/K)$.

As in \eqref{decom-0}, the $K$-invariance of $\varphi_\lambda$,  $\lambda\in \mathfrak a_0^\ast$ yields,
\begin{align*}
\what{h}(\lambda) &=\int_G (d_r(u)f)_0(x)\  \varphi_{-\lambda}(x) \, dx
=\int_G f(x)\ \ (d_r(u^\ast) \varphi_{-\lambda})(x)\, dx.
\end{align*}

From this and  \eqref{derivative-phi} we have,
\begin{equation}
\label{step-2}
\what{h}(\lambda)=\int_G f(x)\ \ \sum_{i\in \Lambda}\  P_i(\lambda) \, \varphi_{-\lambda}^i(x) dx
=\sum_{i\in \Lambda} P_i(\lambda)\what{f}_i(\lambda),
\end{equation}
where $\Lambda$ is a finite set, $P_i$ are polynomials and $\what{f}_i(\lambda)$ are
  as defined in \eqref{FT-i}.

Since $\Delta^m f\in L^p(G/K)$ for all $m\in \N\cup\{0\}$, by \eqref{HY-general} we have  \[(|\lambda|^2+|\rho|^2)^m\what{f}_i\in L^{p'}(\mathfrak a_0^\ast, |\hc(\lambda)|^{-2} d\lambda).\]  Hence for any polynomial $P(\lambda)$, $P(\lambda) \what{f}_i\in L^{p'}(\mathfrak a_0^\ast, |\hc(\lambda)|^{-2} d\lambda)$. This shows that $\what{h}\in L^{p'}(\mathfrak a_0^\ast, |\hc(\lambda)|^{-2} d\lambda)$.
\vspace{.2in}

\noindent{\em Step 3:} Using the function $h$ in Step 2, define a Borel Measure $\mu$ on $\mathfrak a_0^\ast$ by
\[\mu(E)=\int_E|\what{h}(\lambda)| |\hc(\lambda)|^{-2} d\lambda,\] for Borel subsets $E$ of $\mathfrak a_0^\ast$.
Then for any $m\in \N\cup\{0\}$ and $\alpha$ a large positive integer,
\begin{align*} &\int_{\mathfrak a_0^\ast} (|\lambda|^2+\rho|^2)^m \  |\what{h}(\lambda)|\ |\hc(\lambda)|^{-2} d\lambda\\
=&\int_{\mathfrak a_0^\ast} (|\lambda|^2+|\rho|^2)^{m+\alpha}\ \ |\what{h}(\lambda)|\ \ \frac 1{(|\lambda|^2+\rho|^2)^{\alpha}}\ \  |\hc(\lambda)|^{-2} d\lambda\\
\le&\sum_{i\in \Lambda} \int_{\mathfrak a_0^\ast} (|\lambda|^2+|\rho|^2)^{m+\alpha}\ \ |\what{f}_i(\lambda)|\ \ \frac {|P_i(\lambda)|}{(|\lambda|^2+\rho|^2)^{\alpha}}\ \ |\hc(\lambda)|^{-2} d\lambda\\
\le &\sum_{i\in \Lambda} \left( \int_{\mathfrak a_0^\ast} (|\lambda|^2+|\rho|^2)^{(m+\alpha)p'} |\what{f}_i(\lambda)|^{p'} |\hc(\lambda)|^{-2} d\lambda \right)^{1/p'} C_{\alpha, p,i}\\
\le &\sum_{i\in \Lambda} C_{\alpha,p, i} \|\Delta^{m+\alpha} f\|_p\\
\le & C_{\alpha,p} \|\Delta^{m+\alpha} f\|_p.
\end{align*}
Here,  by \eqref{plancherel-estimate},
\[C_{\alpha, p, i} =  \left( \int_{\mathfrak a_0^\ast}  \frac {|P_i(\lambda)|^p}{(|\lambda|^2+\rho|^2)^{p\alpha}}\ \ |\hc(\lambda)|^{-2} d\lambda \right)^{1/p}<\infty\] for $\alpha$ suitably large and $C_{\alpha, p}=\sum_{i\in \Lambda} C_{\alpha, p, i}$.
In the inequality above, we have used \eqref{step-2} in the second step, H\"older's inequality in the third step and Hausdorff-Young inequality \eqref{HY-general} in the fourth step.
Above we have assumed $p\in (1, 2]$, but for $p=1$ it can be appropriately modified to yield the same result. Thus we have established that for any $m\in \N\cup\{0\}$ and $\alpha$ a large positive integer,
\begin{equation}\label{eq-1}
\int_{\mathfrak a_0^\ast} (|\lambda|^2+\rho|^2)^m \  |\what{h}(\lambda)|\ |\hc(\lambda)|^{-2} d\lambda\le C_{\alpha,p} \|\Delta^{m+\alpha} f\|_p, \ \ p\in [1,2].
\end{equation}
In particular $\what{h}\in L^1(\mathfrak a_0^\ast, |\hc(\lambda)|^{-2} d\lambda)$ and hence $\mu$ is a finite measure. Since
 $\what{h}$ is also in $L^{p'}(\mathfrak a_0^\ast, |\hc(\lambda)|^{-2} d\lambda)$ (see Step 2), we have $\what{h}\in L^2(\mathfrak a_0^\ast, |\hc(\lambda)|^{-2} d\lambda)$ and hence,
\[\int_{\mathfrak a_0^\ast}|\what{h}(\lambda)|d\mu(\lambda)=\int_{\mathfrak a_0^\ast}|\what{h}(\lambda)|^2 |\hc(\lambda)|^{-2} d\lambda <\infty.\] That is, $\what{h}\in L^1(\mathfrak a_0^\ast, d\mu)$ and due to its $W$-invariance $\what{h}\in L^1(\mathfrak a_0^\ast, d\mu)^W$.

 Since for any $D\in \mathbf{D}(G/K)$, $D\varphi_\lambda=\Gamma(D)(i\lambda) \varphi_\lambda$ and $Dh$ is $K$-biinvariant as $h$ is so, we have
\[\what{Dh}(\lambda)= P(\lambda) \what{h}(\lambda)\] for some   $W$-invariant polynomial $P(\lambda)$.
Taking $m$ appropriately large in the inequality \eqref{eq-1} we see that
$\what{Dh}\in L^1(\mathfrak a_0^\ast, |\hc(\lambda)|^{-2} d\lambda)$ for all $D\in \mathbf{D}(G/K)$. Hence, by the inversion formula \eqref{inversion} we get:
 \[Dh(o)=|W|^{-1}\int_{\mathfrak a_0^\ast} P(\lambda)\, \what{h}(\lambda)\, |\hc(\lambda)|^{-2} d\lambda.\]
We have noted at the beginning of Step 2 that $Dh(o)=0$ for any $D\in \mathbf{D}(G/K)$. Thus
\begin{equation} \label{where-0}
\int_{\mathfrak a_0^\ast} P(\lambda)\, \what{h}(\lambda)\, |\hc(\lambda)|^{-2} d\lambda = 0,
\end{equation} for any $W$-invariant polynomial $P$.
\vspace{.2in}

\noindent{\em Step 4:} The function $h$ is as  in Step 2 and Step 3. We fix an orthonormal basis  of $\mathfrak a_0^\ast$ and write $\lambda\in \mathfrak a_0^\ast$ as $\lambda=(\lambda_1, \lambda_2, \ldots, \lambda_l)$ using this basis. For $j=1,2,\ldots, l$ we define the sequence of moments $S_j(m)$ as
\[S_j(m)=\int_{\mathfrak a_0^\ast} |\lambda_j|^m d\mu(\lambda).\]
Then by definition of $\mu$ and  \eqref{eq-1}, for suitably large $\alpha\in \N$,
\begin{align*} S_j(2m) \le&\int_{\mathfrak a_0^\ast}(|\lambda|^2+|\rho|^2)^m\, |\what{h}(\lambda)|\, |\hc(\lambda)|^{-2} d\lambda\\
\le& C_{p, \alpha} \|\Delta^{m+\alpha} f\|_p,
\end{align*} for some constant $C_{p,\alpha}$.
From condition (ii) in the hypothesis and Lemma \ref{BPR1-3.3}, we get
\[\sum_{m=1}^\infty S_j(2m)^{-1/2m}=\infty.\]
Therefore by Lemma \ref{Carleman-lem}, given $\epsilon>0$, there exists a $W$-invariant polynomial $P_\epsilon$ such that $\|\bar{\what{h}}-P_\epsilon\|_{L^1(\mathfrak a_0^\ast, d\mu)}<\epsilon$.

Hence,
\begin{align*}
&\int_{\mathfrak a_0^\ast} |\what{h}(\lambda)|^2 |\hc(\lambda)|^{-2} d\lambda\\
=& \int_{\mathfrak a_0^\ast} (\overline{\what{h}(\lambda)} - P_\epsilon(\lambda) + P_\epsilon(\lambda))\,\, \what{h}(\lambda)\,\, |\hc(\lambda)|^{-2} d\lambda\\
\le& \int_{\mathfrak a_0^\ast} |(\overline{\what{h}(\lambda)} - P_\epsilon(\lambda)|\, \, d\mu(\lambda)+ \left|\int_{\mathfrak a_0^\ast} P_\epsilon(\lambda) \what{h}(\lambda) |\hc(\lambda)|^{-2} d\lambda\right|\\
<& \epsilon,
\end{align*}
as the second integral in the last but one line is $0$ by \eqref{where-0}. This proves $\what{h}=0$ and hence $h=0$ as an $L^p$-function.

Since $h=(d_r(u)f)_0$ is defined in Step 2 using arbitrary $u\in U(\g)$, in view of Lemma \ref{lem-1}, $f=0$ and the proof is completed.
\end{proof}


\begin{remark} \label{reduce-to-biinvariant}
Theorem \ref{BPR2}   is essentially accommodated in Theorem \ref{main-result}.
To keep the argument simple let us first take $x_0$ to be the origin. We assume that  $f$ in Theorem \ref{main-result} is $K$-biinvariant. Take  $u\in U(\g)^K\subset U(\g)$. Then $d_r(u)f=d_l(u)f=D(u)f$. In general, for $u\in U(\g)$,  $\tilde{u}=\int_K Ad(k) u\, dk \in U(\g)^K$ and it follows from the definition that \[d_r(u)f(o)= d_r(\tilde{u})f(o)= d_l(\tilde{u})f(o)= D(\tilde{u})f(o),\] using $K$-biinvariance of $f$.  Thus the hypothesis of Theorem \ref{main-result} reduces to that of Theorem \ref{BPR2} when $f$ is assumed to be $K$-biinvariant.

If $x_0=gK$ in Theorem \ref{BPR2} is not the origin,  then note that the $K$-biinvariant function $(\ell_g f)_0$ satisfies its  hypothesis with $x_0$ replaced by the origin, hence by the first part of the argument $(\ell_gf)_0=0$. To complete this discussion we only need to note that, since $f$ is $K$-biinvariant, $(\ell_gf)_0 =0$ implies that $f=0$. Precisely, $(\ell_g f)_0=0$ implies that its spherical Fourier transform, which is $\what{f}(\lambda\varphi_\lambda(g)$ for almost every  $\lambda\in \ag_0^\ast$, is zero. Hence $\what{f}(\lambda)=0$, for almost every $\lambda\in \ag_0^\ast$, since $\lambda\mapsto \varphi_\lambda(g)$ is a real analytic function on $\ag_0^\ast$.

\end{remark}
\section{Closing Comments}
1. A comparison between  Theorem \ref{main-result} and Theorem \ref{chernoff} might be worth consideration. To put Theorem \ref{chernoff} in perspective, we view  $\R^n$  as the symmetric space $M(n)/\mathrm{SO}(n)$, where $M(n)=\R^n\rtimes  \mathrm{SO}(n)$ is the group of isometries of $\R^n$. The condition (i) and (ii) in the hypothesis in Theorem \ref{main-result} are evidently similar to the first two conditions in Theorem \ref{chernoff}. But we may need to elaborate on condition (iii).  While Theorem \ref{chernoff} uses all partial derivatives of $\R^n$ for the vanishing condition,  we use all differential operators of $G$ which preserves the right $K$-invariance of the function and thereby are precisely the set of differential operators relevant for functions on $G/K$. A specific question could be that if the differential operators induced by an element of $\mathfrak{so}(n)$ is used in Theorem \ref{chernoff}. Indeed, a straightforward computation shows that for any $A=(A_{ij})\in \mathfrak{so}(n)$ and $f\in C^\infty(\R^n)$,
\[f(A;x)=\frac d{dt}f(\exp tAx)|_{t=0}=\sum_{i=1}^n\sum_{j=1}^n A_{ij}\, x_j\, \frac{\partial f}{\partial x_j}(x).\]  Thus the vanishing condition in Theorem \ref{chernoff} (which uses all partial derivatives of $\R^n$) implies that for any differential operator $D$ of $M(n)$, $Df(x_0)=0$. This establishes the analogy.
\vspace{.2in}

2. If in Theorem \ref{chernoff}, we consider the vanishing condition only for the $M(n)$-invariant differential operators, instead of all partial derivatives, the condition reduces to $\Delta_{\R^n}^mf(x_0)=0$ for any $m\in \N \cup \{0\}$ where $\Delta_{\R^n}$ is the usual Euclidean Laplacian. The conclusion will then be that the radial average of $f$ around the point $x_0$ is zero. Analogously,  if our vanishing condition involves only  the left-invariant differential operators, i.e. for all $D\in \mathbf{D}(G/K)$, $Df(x_0)=0$, then we can conclude that $(\ell_{g}f)_0=0$, where $x_0=gK$, $f_0(x)=\int_Kf(kx) dk$ and $\ell_{g}f(x)=f(gx)$. The precise statement is given below, which  is Theorem \ref{BPR2} without the restriction of $K$-biinvariance on the function $f$.
\begin{theorem}\label{writeup-result-2} Fix a $p\in [1,2]$.
Let $f\in C^\infty(G/K)$ be such that $\Delta^m f\in L^p(G/K)$ for all integers $m\ge 0$ and
\[\sum_{m=0}^\infty \|\Delta^m f\|_p^{-1/2m}=\infty.\]
If $Df(x_0)=0$ at some point $x_0=gK\in G/K$ for all $D\in \mathbf{D}(G/K)$, then $(\ell_{g}f)_0=0$.
\end{theorem}

Here is  a sketch of the proof. Since $D\in \mathbf{D}(G/K)$ commutes with translations $f$ can be replaced by $h=\ell_{g}f$, so that the vanishing condition in the hypothesis is at the origin $o$, instead of at $x_0$. That is $Dh(o)=0$ for all $D\in \mathbf{D}(G/K)$. But $Dh(o)=0$ implies that $D h_0(o)=0$.
Other two conditions are also satisfied by  $h_0$, because by Minkowski's inequality, $\|h_0\|_p\le \|f\|_p$ and $\Delta$ commutes with translations. Note that $h_0$ is $K$-biinvariant.
Thus by Theorem \ref{BPR2},   $h_0=(\ell_{g}f)_0=0$. \qed

A counter example is given in \cite{BPR1} to show that the hypothesis of this theorem does not imply that $f=0$.
Crux is that for  any function $f$ on $G/K$ whose
$K$-biinvariant component is $0$, $Df(o)=0$ for all $D\in \mathbf{D}(G/K)$. In other words, for a function $f \in C^\infty(G/K)$ and for any $D\in \mathbf{D}(G/K)$, the condition $Df(o)=0$  is equivalent to $D f_0(o)=0$, and therefore  such a vanishing condition imposes nothing on the function $f-f_0$.
\vspace{.2in}

\noindent{\bf Acknowledgements:}  I would like to thank Swagato K Ray  for introducing me to this problem. I am grateful to him, to Kingshook Biswas and Jyoti Sengupta for many illuminating discussions.


\end{document}